\documentclass[a4paper,11pt]{article}
\usepackage{amssymb,enumerate}
\usepackage{amsmath,amsthm}
\usepackage{amsfonts}
\usepackage{graphicx}

\newcounter{case}

\renewcommand{\thecase}{\arabic{case}}

\usepackage{color}
\newtheorem{theorem}{Theorem}[section]
\newtheorem{proposition}[theorem]{Proposition}
\newtheorem{lemma}[theorem]{Lemma}
\newtheorem{corollary}[theorem]{Corollary}

\newtheorem{example}[theorem]{Example}

\newcommand{\Aut}{\hbox{{\rm Aut}}}
\newcommand{\Fix}{\hbox{{\rm Fix}}}
\newcommand{\Dih}{\hbox{{\rm Dih}}}
\newcommand{\ZZ}{\mathbb{Z}}

\newcommand{\Cay}{\mathrm{Cay}}

\theoremstyle{definition}
\newtheorem{definition}[theorem]{Definition}
\begin{document}

\begin{center}
{\bf\large On automorphisms and structural properties\\ of generalized Cayley graphs} \\ [+4ex]

Ademir Hujdurovi\'c{\small$^{a,b}$},\
Klavdija Kutnar{\small$^{a,b,}$}\footnote{Corresponding author: Klavdija Kutnar (e-mail: klavdija.kutnar@upr.si).} \\ [+2ex]
Pawe{\l} Petecki{\small$^{a,b,c}$} \ and\
Anastasiya Tanana{\small$^{a,d}$}\\ [+2ex]
{\it \small $^a$University of Primorska, UP FAMNIT, Glagolja\v ska 8, 6000 Koper, Slovenia\\
$^b$University of Primorska, UP IAM, Muzejski trg 2, 6000 Koper, Slovenia\\
$^c$AGH University of Science and Technology, al. Mickiewicza 30, 30-059 Krakow, Poland\\
$^d$University of Bonn, Regina-Pacis-Weg 3, D-53113 Bonn, Germany
 }

\end{center}

\begin{abstract}
In this paper, generalized Cayley graphs are studied. It is proved that every generalized Cayley graph of order $2p$ is a Cayley graph, where $p$ is a prime. Special attention is given to generalized Cayley graphs on Abelian groups. It is proved that every generalized Cayley graph on an Abelian group with respect to an automorphism which acts as inversion is a Cayley graph if and only if the group is elementary Abelian $2$-group, or its Sylow $2$-subgroup is cyclic. Necessary and sufficient conditions for a generalized Cayley graph to be unworthy are given.
\end{abstract}

\begin{quotation}
\end{quotation}

\section{Introduction}\label{sec:intro}
\noindent

In this paper we consider generalized Cayley graphs, first introduced   in \cite{MSS92}.
\begin{definition}\label{def:gen_cayley}
Let $G$ be a group, $S$ a subset of $G$ and $\alpha$ an automorphism of $G$
such that the following conditions are satisfied:
\begin{enumerate}[(i)]
\item $\alpha^2=1$,
\item if $x\in G$ then $\alpha(x^{-1})x \not \in S$,
\item if $x,y\in G$ and $\alpha(x^{-1})y\in S$ then $\alpha(y^{-1})x\in S$.
\end{enumerate}
Then the {\em generalized Cayley graph} $X=GC(G,S,\alpha)$ on $G$ with respect to the ordered pair
$(S,\alpha)$ is the graph with vertex set $G$, with two vertices $x,y\in V(X)$ being adjacent in $X$
if and only if $\alpha(x^{-1})y\in S$. 
In other words, a vertex $x\in G$ is adjacent to all the vertices of the
form $\alpha(x)s$, where $s\in S$.
\end{definition}

Note that (ii) implies that $X$ has no loops, and (iii) implies that $X$ is undirected.
Also, in view of (i), the condition (iii) is equivalent to $\alpha(S^{-1})=S$.
Namely, by letting $x=1$ in (iii), we obtain $\alpha(S^{-1})=S$, and conversely,
if $\alpha(S^{-1})=S$, then $\alpha(x^{-1})y\in S$ implies that $\alpha(y^{-1}\alpha(x))=\alpha(y^{-1})x\in S$.
If $\alpha=1$ then we say that $GC(G,S,\alpha)$ is a {\em Cayley graph} and write simply $\Cay(G,S)$.
Therefore every Cayley graph is also a generalized Cayley graph, but the converse is not true (see
\cite[Proposition~3.2]{MSS92}).
A generalized Cayley graph $GC(G,S,\alpha)$ is connected if and only if $S$ is a left generating set for the quasigroup $(G,*)$, where $f*g=\alpha(f)g$ for all $f,g\in G$ (see \cite[Proposition~3.5]{MSS92}). 

In \cite{MSS92} the properties of generalized Cayley graphs
relative  to canonical  double  covers (also called bipartite double covers)  of  graphs are considered. For graphs $X$ and $Y$ the {\it direct product $X\times Y$} of $X$ and $Y$ is the graph with vertex set $V(X\times Y)=V(X)\times V(Y)$, and two vertices $(x_1,y_1)$ and $(x_2,y_2)$ are adjacent in $X\times Y$ if and only if $x_1$ is adjacent with $x_2$ in $X$ and $y_1$ is adjacent with $y_2$ in $Y$.
{\it Canonical double cover} $B(X)$ of a graph $X$ is the direct product $X \times K_2$ ($K_2$ is the complete graph on two vertices). 
It is easily seen that $\Aut(B(X))$ contains a subgroup isomorphic to $\Aut(X)\times \ZZ_2$. If $\Aut(B(X))$ is isomorphic to $\Aut(X)\times \ZZ_2$ then the graph $X$ is called {\it stable}, otherwise it is called {\it unstable}. This concept was first defined by Maru\v si\v c et al. \cite{MSS89} and studied later most notably by Surowski \cite{S01,S03}, Wilson \cite{W08}, Lauri et al. \cite{LMS15}. 

If $X$ is a connected non-bipartite graph, then  $B(X)$ is a Cayley graph if and only if $X$ is a generalized Cayley graph (\cite[Theorem~3.1]{MSS92}). It is also proved that every generalized Cayley graph which is stable is a Cayley graph (see \cite[Proposition~3.3]{MSS92}). Therefore, every generalized Cayley graph which is not Cayley graph is unstable. Recently it was proved that there are infinitely many vertex-transitive generalized Cayley graphs which are not Cayley graphs
 (see \cite{HKM15}).

In this paper we continue studying the properties of generalized Cayley graphs started in \cite{HKM15}. Since for defining generalized Cayley graphs one needs a group automorphism of order two, in Section~\ref{sec:aut} we study some properties of such automorphisms. 
In Section~\ref{sec:inversion} we study generalized Cayley graphs on Abelian groups arising from the inversion automorphism $\iota:x\mapsto x^{-1}$. We prove that every generalized Cayley graph on an Abelian group $G$ with respect to $\iota$ is a Cayley graph if and only if $G$ is either elementary Abelian 2-group, or Sylow 2-subgroup of $G$ is cyclic (see Theorem~\ref{thm:maininversion}).
In  Section~\ref{sec:2p} it is proved that every generalized Cayley graph of order  twice a prime 
is a Cayley graph.  Necessary and sufficient conditions for a generalized Cayley graph to be unworthy are given in Section 5 (a graph is called {\it unworthy} if it has two vertices with the same neighbours).


\section{Group automorphisms of order two}\label{sec:aut}
\noindent

For defining a generalized Cayley graph on a group $G$, one needs an automorphism of $G$ of order two. Therefore it is important to understand the structure of  such group automorphisms. We start this section with the following proposition, which tells us that for studying generalized Cayley graph on a group $G$, it is sufficient to consider only the representatives of conjugacy classes in $\Aut(G)$. 

\begin{proposition}\label{prop:conjugation}
$GC(G,S,\alpha)\cong GC(G,\varphi(S),\varphi \alpha \varphi^{-1})$ for any $\varphi \in \Aut(G)$.
\end{proposition}
\begin{proof}
Let us first prove that $\varphi(S)$ and $\varphi \alpha \varphi^{-1}$ satisfy the conditions from Definition~\ref{def:gen_cayley}.
\begin{enumerate}[(i)]
\item $(\varphi \alpha \varphi^{-1})^2=\varphi \alpha \varphi^{-1}\varphi \alpha \varphi^{-1}=id$. 
\item Suppose that there exists $g\in G$ such that $(\varphi \alpha \varphi^{-1})(g^{-1})g\in \varphi(S)$. 
Then $\varphi^{-1}((\varphi \alpha \varphi^{-1})(g^{-1})g)\in S$, and consequently $\alpha((\varphi^{-1}(g))^{-1})\varphi^{-1}(g)\in S$, a contradiction.
\item $(\varphi \alpha \varphi^{-1})(\varphi(S))=\varphi (\alpha(S))=\varphi(S^{-1})=\varphi(S)^{-1}$.
\end{enumerate}

Mapping $\varphi$ is clearly a bijective mapping from $G$ to $G$. Let $\{x,y\}$ be an arbitrary edge in $GC(G,S,\alpha)$. Then $y=\alpha(x)s$ for some $s\in S$. Further we have
$$\varphi(y)=\varphi(\alpha(x)s)=\varphi(\alpha(x))\varphi(s)=(\varphi\alpha)(x)\varphi(s)=(\varphi\alpha \varphi^{-1})(\varphi(x))\varphi(s).$$ 
This implies that $\varphi(x)$ and $\varphi(y)$ are adjacent in $GC(G,\varphi(S),\varphi \alpha \varphi^{-1})$. Similarly, one can see that if $\varphi(x)$ and $\varphi(y)$ are adjacent in $GC(G,\varphi(S),\varphi \alpha \varphi^{-1})$, then $x$ and $y$ are adjacent in $GC(G,S,\alpha)$.
Therefore $\varphi$ is in fact an isomorphism between $GC(G,S,\alpha)$ and $GC(G,\varphi(S),\varphi \alpha \varphi^{-1})$.
\end{proof}

For a group $G$ and $\alpha \in \Aut(G)$, the set $\Fix(\alpha)$ is defined as $\Fix(\alpha)=\{g\in G \mid \alpha(g)=g\}$.
We let $\omega_\alpha\colon G\to G$ be the mapping defined by
$\omega_{\alpha}(x)=\alpha(x)x^{-1}$ and let $\omega_{\alpha}(G)=\{\omega_{\alpha}(g)\mid g \in G\}$. Notice that Definition~\ref{def:gen_cayley}(ii) is equivalent to $\omega_{\alpha}(G)\cap S=\emptyset$. In the following proposition  some properties of these sets are given. The proof is straightforward and is omitted.

\begin{proposition}\label{prop:claims}
Let $G$ be a group and $\alpha \in \Aut(G)$ such that $\alpha^2=1$. Then the following hold:
\begin{enumerate}[(a)]
\item $\Fix(\alpha)$ is a subgroup of $G$;
\item  If $G$ is an Abelian group then $\omega_{\alpha}(G)$ is a subgroup of $G$;
\item $\alpha(x)=x^{-1}$, for every $x\in \omega_{\alpha}(G)$.
\end{enumerate}
\end{proposition}

Observe that the set $\omega_{\alpha}(G)$ does not need to be a subgroup of $G$ in general. For example, if $G=A_4$ and $\alpha(g)=(1\,2)g(1\,2)$, then $\omega_{\alpha}(G)=\{id,(1\,2\,3),(1\,3\,2),(1\,2\,4),(1\,4\,2),(1\,2)(3\,4)\}$, which is not a subgroup of $A_4$.
The following result is proved in \cite{HKM15} and will be used later.

\begin{lemma}\label{lem:fix_alpha}
{\rm \cite{HKM15}} Let $G$ be a group and $\alpha \in \Aut(G)$. Then 
$
\displaystyle |\Fix(\alpha)| |\omega_{\alpha}(G)|=|G|.
$
\end{lemma}

The following result is proved in \cite{M09}, however because of completeness and some differences in terminologies used we include the proof here.
\begin{proposition}\label{prop:odd} {\rm \cite{M09}}
If $G$ is an Abelian group of odd order then $G=\Fix(\alpha)\times \omega_{\alpha}(G)$.
\end{proposition}
\begin{proof}
Let $x\in \Fix(\alpha)\cap \omega_{\alpha}(G)$. Since $x\in \omega_{\alpha}(G)$, by Proposition~\ref{prop:claims}(c), we have $\alpha(x)=x^{-1}$. 
On the other hand, since $x\in \Fix(\alpha)$, we have $\alpha(x)=x$, and therefore $x=x^{-1}$. Since $G$ is a group of odd order, we can conclude that $x=1_G$.
Therefore, $|\Fix(\alpha)\cap \omega_{\alpha}(G)|=1$, 
which implies 
$$
|\Fix(\alpha)\omega_{\alpha}(G)|=
\frac{|\Fix(\alpha)|\cdot |\omega_{\alpha}(G)|}{|\Fix(\alpha)\cap \omega_{\alpha}(G)|}=
|G|,
$$ 
and therefore $G=\Fix(\alpha)\times \omega_{\alpha}(G)$.
\end{proof}

Proposition~\ref{prop:odd} enables us to describe generalized Cayley graphs on an Abelian group of odd order in the following way. 

\begin{proposition}
Let $X$ be a generalized Cayley graph on an Abelian group of odd order. Then $X$ is isomorphic to the graph $Y$ given with
\begin{enumerate}
\item $V(Y)=G_1\times G_2$, where $G_1$ and $G_2$ are Abelian groups of odd order;
\item $E(Y)=\{\{(g_1,g_2),(g_1s_1,g_2^{-1}s_2)\}\mid (g_1,g_2)\in G_1\times G_2,~ (s_1,s_2) \in \overline{S}\}$
where $\overline{S}\subseteq (G_1\setminus \{1\})\times G_2$ such that 
 $(s_1,s_2)\in \overline{S}\Leftrightarrow (s_1^{-1},s_2)\in \overline{S}$.
\end{enumerate}
\end{proposition}
\begin{proof}
Let $X=GC(G,S,\alpha)$ be a generalized Cayley graph on Abelian group $G$ of odd order. By Proposition~\ref{prop:odd} $G=G_1\times G_2$,  where $G_1=\Fix(\alpha)$ and $G_2=\omega_{\alpha}(G)$. Let $\varphi$ be a natural isomorphism between $G$ and $G_1\times G_2$, that is $\varphi(g)=(g_1,g_2)$, where $g=g_1g_2$, $g\in G, g_1\in G_1, g_2\in G_2$.
Let $\overline{S}=\varphi(S)$.
Recall that by Definition~\ref{def:gen_cayley}(ii) $S\cap \omega_{\alpha}(G)=\emptyset$. Hence, for every $s\in S$, we have $s=s_1s_2$, where $s_1\in G_1\setminus\{1\}$ and $s_2\in G_2$, implying that $\overline{S}\subseteq (G_1\setminus \{1\})\times G_2$.
Also, since $s\in S\Leftrightarrow \alpha(s^{-1})\in S$, it follows that $(s_1,s_2)\in \overline{S}\Leftrightarrow (s_1^{-1},s_2)\in \overline{S}$.

Let $\{g,\alpha(g)s\}$ be an edge of $X$. Then $g=g_1g_2$, and $s=s_1s_2$, where $g_1,s_1\in G_1$, $g_2,s_2\in G_2$ and $s\in S$. This further implies that $\alpha(g)s=\alpha(g_1g_2)s_1s_2=g_1g_2^{-1}s_1s_2=(g_1s_1)(g_2^{-1}s_2)$. Hence we obtain $\varphi(g)=(g_1,g_2)$ and $\varphi(\alpha(g)s)=(g_1s_1,g_2^{-1}s_2)$. Therefore, $\varphi$ is an isomorphism between $X$ and $Y$.
\end{proof}

In the case of Abelian groups of even order we do not have a result analogues to Proposition~\ref{prop:odd}. However, we do have a similar result for Abelian groups with cyclic Sylow 2-subgroup.
\begin{proposition}
Let $G$ be an Abelian group with cyclic Sylow 2-subgroup of order $2^n$ and let $\alpha\in \Aut(G)$  be such that $\alpha^2=1$. Let $H$ be the subgroup of $G$  of index $2^n$. Let $H_1=\Fix(\alpha)\cap H$ and $H_2=\omega_{\alpha}(G)\cap H$. Then $G=\ZZ_{2^n}\times H_1\times H_2$ and there exists $a\in \{\pm1, 2^{n-1}\pm 1\}$ such that for $(x,y_1,y_2)\in\ZZ_{2^n}\times H_1\times H_2 $, $\alpha(x,y_1,y_2)=(ax,y_1,y_2^{-1})$.
\end{proposition}
\begin{proof}
Let $\alpha\in \Aut(G)$ be such that $\alpha^2=1$. Since $H$ is of index $2^n$,  $H$ is of odd order, and hence $G= \ZZ_{2^n}\times H$. Therefore, $\alpha$ can be written as a product of an automorphism $\alpha_1$ of $\ZZ_{2^n}$ that fixes $H$, and an automorphism $\alpha_2$ of $H$ that fixes $\ZZ_{2^n}$, that is $\alpha=\alpha_1 \alpha_2$, where $\alpha_1\in \Aut(\ZZ_{2^n})$ and $\alpha_2\in \Aut(H)$, such that $\alpha_1^2=\alpha_2^2=1$. 
Since $\alpha_1\in \Aut(\ZZ_{2^n})$, we have $\alpha_1(x)=ax$, where $1\leq a<2^n$ is odd. The condition $\alpha_1^2=1$, implies that $a^2\equiv 1\pmod{2^n}$, and it is not difficult to verify that $a\in \{\pm 1, 2^{n-1}\pm 1\}$.

By Proposition~\ref{prop:odd}, $H=\Fix(\alpha_2)\times \omega_{\alpha_2}(H)$. Since  $\alpha(x)=\alpha_2(x)$ for every $x\in H$,  it is not difficult to see that $\Fix(\alpha_2)=\Fix(\alpha)\cap H=H_1$ and $\omega_{\alpha_2}(H)=\omega_{\alpha}(G)\cap H=H_2$. We conclude that $G=\ZZ_{2^n}\times H_1\times H_2$. Proposition~\ref{prop:odd} also implies that for $(y_1,y_2)\in H_1\times H_2$, $\alpha_2(y_1,y_2)=(y_1,y_2^{-1})$. Since $\alpha=\alpha_1\alpha_2$, the result follows.
\end{proof}


\section{Generalized Cayley graphs with respect to the inversion automorphism}
\label{sec:inversion}
\noindent

Throughout this section we assume that $G$ is an Abelian group, and that $\iota$ is the inversion automorphism of $G$, that is $\iota(x)=x^{-1}$ ($\forall x \in G$).
Before stating the main results of this section, let us recall the definition of generalized dihedral groups. For an Abelian group $G$, the {\it generalized dihedral group} $\Dih(G)= G \rtimes \ZZ_{2}$, with $\ZZ_2$ acting on $G$ by inverting elements. More precisely, for $i\in \ZZ_{2}$ and $g_1,g_2\in G$
\begin{center}
$\begin{array}{ccl}
(g_1,i)  \cdot (g_2,0) & = & (g_1g_2,i)\\
(g_1,i)\cdot(g_2,1)&=&(g_1^{-1}g_2,i+1).
\end{array}$
\end{center}
(Note that in a standard definition of the generalized dihedral group  $\Dih(G)= G \rtimes \ZZ_{2}$
the group operation is defined by
$(g_1,i)\cdot(g_2,j)=(g_1\varphi(i)(g_2),i+j)$, where $\varphi\colon \ZZ_2\to \Aut(G)$ is a group homomorphism
mapping $0\in \ZZ_2$ to the identity automorphism of $G$ and mapping $1\in \ZZ_2$ to $\iota\in\Aut(G)$. However, it is not difficult
to check that our definition is equivalent to this standard definition.)

The following theorem shows that generalized Cayley graphs on particular Abelian groups of even order
with respect to the inversion automorphism are isomorphic to Cayley graphs on generalized dihedral groups.

\begin{theorem}\label{thm:AbelInverse}
Let $n$ be a positive integer and let $G$ be a finite Abelian group of odd order. Then the generalized Cayley graph $GC(\mathbb{Z}_{2^n}\times G, S,\iota)$ is isomorphic to a Cayley graph on $\Dih(\mathbb{Z}_{2^{n-1}}\times G)$.
\end{theorem} 

\begin{proof}
 Let  $X=GC(\mathbb{Z}_{2^n}\times G, S,\iota)$. 
It is clear that $( (\mathbb{Z}_{2^{n-1}}\times G)\times \mathbb{Z}_2,\cdot)$ 
with the operation $\cdot$ defined in the following way:
\begin{eqnarray*}
((x_1,g_1),i)\cdot ((x_2,g_2),0)&=&((x_1+x_2,g_1g_2),i)\\
((x_1,g_1),i)\cdot ((x_2,g_2),1)&=&((-x_1+x_2,g_1^{-1}g_2),i+1)
\end{eqnarray*}
is the generalized dihedral group $\Dih(\mathbb{Z}_{2^{n-1}}\times G)$.
(The group $G$ is Abelian, but we write its operation in multiplicative way.)
Observe that  $((x,g),0)^{-1}=((-x,g^{-1}),0)$ and 
$((x,g),1)^{-1}=((x,g),1)$.
It is also not difficult to verify that for $x_1,x_2\in \ZZ_{2^{n-1}}$ and $g_1,g_2\in G$,
\begin{equation}\label{eq:product}
((x_1,g_1),0)^{-1}\cdot ((x_2,g_2),1)=((x_1,g_1),1)^{-1}\cdot ((x_2,g_2),0)=((x_1+x_2,g_1g_2),1)
\end{equation}

 
For $(x,g)\in \mathbb{Z}_{2^n}\times G$, define a mapping $\varphi :\mathbb{Z}_{2^n}\times G\rightarrow  (\mathbb{Z}_{2^{n-1}}\times G)\times \mathbb{Z}_2$ with
$$\varphi ((x,g))=\left(\left(\left\lfloor \frac{x}{2}\right\rfloor, g\right),x\ \text{mod}\ 2\right).$$

It is not difficult to verify that $\varphi$ is bijection. 
Definition~\ref{def:gen_cayley}(ii) implies that for any $(x,g)\in \mathbb{Z}_{2^n}\times G$, $\iota\left((x,g)^{-1}\right)\cdot(x,g)=(x,g)\cdot(x,g)=(2x,g^2)\not \in S$. 
Since $G$ is of odd order we have $G=\{g^2\mid g\in G\}$.  Therefore, if $s=(x,g)\in S$, then $x$ is odd and $\varphi (s)=\left(\left(\frac{x-1}{2},g \right),1\right)$.

We claim that $\varphi$ is isomorphism between $X$ and $\Cay(\Dih(\ZZ_{2^{n-1}}\times G),\varphi(S))$.
Let $(x_1,g_1)$ and  $(x_2,g_2)$ be two adjacent vertices of $X$. Then $(x_1+x_2,g_1g_2)\in S$. We have already seen that for each element in $S$ its first coordinate is an odd element from $\ZZ_{2^n}$, hence either $(x_1\equiv 0\ \text{mod}\ 2, x_2\equiv 1\ \text{mod}\ 2)\ \text{or}\ (x_1\equiv 1\ \text{mod}\ 2, x_2\equiv 0\ \text{mod}\ 2)$. 

Consider first the case when $x_1\equiv 0\ \text{mod}\ 2$ and  $x_2\equiv 1\ \text{mod}\ 2$. 
Observe that now $ \left\lfloor \frac{x_1}{2} \right\rfloor + \left\lfloor \frac{x_2}{2}\right\rfloor= \left\lfloor \frac{x_1+x_2}{2} \right\rfloor$.
Then $\varphi((x_1,g_1))=\left( \left( \left\lfloor \frac{x_1}{2} \right\rfloor ,g_1 \right), 0 \right)$ and $\varphi((x_2,g_2))=\left( \left( \left\lfloor \frac{x_2}{2} \right\rfloor ,g_2 \right), 1 \right)$ and using (\ref{eq:product}), we obtain
\begin{multline*}
\varphi((x_1,g_1))^{-1}\cdot\varphi((x_2,g_2))=\left( \left( \left\lfloor \frac{x_1}{2} \right\rfloor ,g_1 \right), 0 \right)^{-1}\cdot \left( \left( \left\lfloor \frac{x_2}{2} \right\rfloor ,g_2 \right), 1 \right) =\\
\left( \left( \left\lfloor \frac{x_1}{2} \right\rfloor +  \left\lfloor \frac{x_2}{2} \right\rfloor ,g_1 g_2 \right), 1 \right)=
  \left( \left( \left\lfloor \frac{x_1+x_2}{2} \right\rfloor  ,g_1g_2 \right), 1 \right)
=\varphi((x_1+x_2,g_1g_2))\in \varphi(S).
\end{multline*}
Similarly, if  $x_1\equiv 1\ \text{mod}\ 2$ and  $x_2\equiv 0\ \text{mod}\ 2$, then again using (\ref{eq:product}), we obtain
\begin{multline*}
\varphi((x_1,g_1))^{-1}\cdot\varphi((x_2,g_2))=\left( \left( \left\lfloor \frac{x_1}{2} \right\rfloor ,g_1 \right), 1 \right)^{-1}\cdot \left( \left( \left\lfloor \frac{x_2}{2} \right\rfloor ,g_2 \right), 0 \right) 
=\varphi((x_1+x_2,g_1g_2))\in \varphi(S).
\end{multline*}
Therefore $\varphi((x_1,g_1))^{-1}\cdot\varphi((x_2,g_2))\in \varphi(S)$ and hence $\varphi(x_1,g_1)$ and $\varphi(x_2,g_2)$ are adjacent in $\Cay(\Dih(\ZZ_2^{n-1}\times G),\varphi(S))$.

Suppose now that $\varphi(x_1,g_1)$ and $\varphi(x_2,g_2)$ are adjacent in $\Cay(\Dih(\ZZ_2^{n-1}\times G),\varphi(S))$ for some $x_1,x_2\in \ZZ_{2^n}$ and $g_1,g_2\in G$. Then $\varphi((x_1,g_1))^{-1}\cdot\varphi((x_2,g_2))\in \varphi(S)$, and we conclude that $x_1$ and $x_2$ have different parity. Further, using (\ref{eq:product}), we obtain that $\varphi((x_1,g_1))^{-1}\cdot\varphi((x_2,g_2))=\varphi((x_1+x_2,g_1g_2))$. Therefore, $(x_1+x_2,g_1g_2)\in S$, and hence $(x_1,g_1)$ and $(x_2,g_2)$ are adjacent in $X$. This shows that $\varphi$ is isomorphism between $X$ and $\Cay(\Dih(\ZZ_2^{n-1}\times G)$.
\end{proof}

Theorem~\ref{thm:AbelInverse} shows that every generalized Cayley graph on  an Abelian group with cyclic Sylow 2-subgroup and with respect to the inversion automorphism is a Cayley graph. The following two examples  show that this does not hold for Abelian groups in general.

\begin{example}\label{ex:2gr}
Let $G=\mathbb{Z}_{2^m}\times \mathbb{Z}_{2^n},\ m\geq 1,\, n\geq 2$, $S=\{(1,0),(0,1),(1,1)\}$ and let $\iota$ be the inversion automorphism of $G$. Then $GC(G,S,\iota)$ is not vertex-transitive.
\end{example}
\begin{proof}
Let us consider the triangles contained in $X=GC(G,S,\iota)$. 
Suppose that the vertices $a,b,c\in G$ form a triangle.  Each edge of this triangle is generated by a different element from $S$, since one element of $S$ generates a perfect matching of the graph.
Therefore, without loss of generality we may assume that
 $a+b=(1,0),\ b+c=(1,1),\ c+a=(0,1)$. From this we obtain $b=(1,0)-a$, $c=(1,1)-b=(0,1)+a$ and $a=(0,1)-c=-a$, and hence $2a=0$. Therefore, each triangle in $X$ contains one element of order $2$ from $G$. If $m\geq 3$ or $n\geq 3$ then the vertex $(2,2)$ does not lie on a triangle, hence $X$ is not vertex-transitive in this case. 
 If $n=2$ and $m\in \{1,2\}$, then it is not difficult to see that all vertices of $X$ does not lie on the same number of triangles. This concludes the proof.  
\end{proof}

\begin{example}\label{ex:Z2 x Z2}
Let $G=\mathbb{Z}_{2}\times \ZZ_2\times \ZZ_{2k+1},\ k\geq 1$, $S=\{(1,0,0),(0,1,0),(1,1,1)\}$ and 
 let $\iota$ be the inversion automorphism of $G$. Then $GC(G,S,\iota)$ is not vertex-transitive.
\end{example}
\begin{proof}
It is easy to verify that the vertex $(0,0,0)$ does not lie on a triangle, whereas the vertex $(0,0,k)$ lies on the triangle $[(0,0,k),(1,0,-k),(0,1,k+1)]$ since $(0,1,k+1)$ is adjacent to  $(0,0,k)$. Therefore $GC(G,S,\iota)$ is not vertex-transitive. 
\end{proof}

Before we state the main result of this section, we need the following lemma.

\begin{lemma}\label{lemma:directproduct}
Let $X=GC(G_1,S_1,\alpha_1)$, $Y=GC(G_2,S_2,\alpha_2)$, and let $\alpha$ be the automorphism of $G_1\times G_2$ defined  by $\alpha(g_1,g_2)=(\alpha_1(g_1),\alpha_2(g_2))$ for $g_1\in G_1$ and $g_2\in G_2$. Then $X\times Y\cong GC(G_1\times G_2,S_1\times S_2,\alpha)$.
\end{lemma}
\begin{proof}
Let us first verify  that $GC(G_1\times G_2, S_1\times S_2,\alpha)$ is well-defined, that is, that $S_1\times S_2$ and $\alpha$ satisfy Definition~\ref{def:gen_cayley}. It is clear that $\alpha$ is an automorphism of $G_1\times G_2$, and that $\alpha^2=1$. 
We have $\omega_{\alpha}(G_1\times G_2)=\left\lbrace \alpha((g_1,g_2)^{-1})(g_1,g_2)\mid g_1\in G_1, g_2\in G_2 \right\rbrace=\omega_{\alpha_1}(G_1)\times \omega_{\alpha_2}(G_2)$. Since $S_1\cap \omega_{\alpha_1}(G_1)=\emptyset$ and $S_2\cap \omega_{\alpha_2}(G_2)=\emptyset$, it follows that $(S_1\times S_2)\cap \omega_{\alpha}(G)=\emptyset$, and therefore Definition~\ref{def:gen_cayley}(ii) is satisfied. It is straightforward to verify that $\alpha(S^{-1})=\alpha_1(S_1^{-1})\times \alpha_2(S_2^{-1})=S_1\times S_2=S$, and consequently Definition~\ref{def:gen_cayley}(iii) is satisfied.

Two vertices $(x_1,y_1)$ and $(x_2,y_2)$ of $X\times Y$ are adjacent in $X\times Y$  if and only if $x_1$ is adjacent to $x_2$ in $X$ and $y_1$ is adjacent to $y_2$ in $Y$. This is equivalent to $\alpha_1(x_1^{-1})x_2\in S_1$ and $\alpha_2(y_1^{-1})y_2\in S_2$, which is equivalent to $\alpha((x_1,y_1)^{-1})(x_2,y_2)\in S_1\times S_2=S$, that is to $\alpha((x_1,y_1)^{-1})(x_2,y_2)\in S$, and this is equivalent to $(x_1,y_1)$ being adjacent to $(x_2,y_2)$ in $GC(G_1\times G_2,S_1\times S_2,\alpha)$. This concludes the proof.
\end{proof}

Theorem~\ref{thm:AbelInverse} shows that every generalized Cayley graph on an Abelian group with cyclic Sylow $2$-subgroup with respect to the inversion automorphism is also a Cayley graph. The same result holds if $G$ is an elementary Abelian 2-group. Namely, in this case the inversion automorphism is the identity mapping, since each element of $G$ is of order 2. 
In the following theorem, we prove that these are the only Abelian groups with this property.

\begin{theorem}\label{thm:maininversion}
Let $G$ be an Abelian group and let $\iota$ be the inversion automorphism of $G$. Then every generalized Cayley graph on $G$ with respect to $\iota$ is a Cayley graph if and only if one of the following holds:
\begin{enumerate}[(i)]
\itemsep=0pt
\item $G$ is an elementary Abelian $2$-group;
\item the Sylow 2-subgroup of $G$ is cyclic.
\end{enumerate} 
\end{theorem}
\begin{proof}
If $G$ is an elementary Abelian $2$-group, then $\iota$ is the identity map, and therefore $GC(G,S,\iota)\cong Cay(G,S)$. 
If the Sylow $2$-subgroup of $G$ is cyclic, then $G\cong \ZZ_{2^n}\times H$, where $H$ is an Abelian group of odd order. By Theorem~\ref{thm:AbelInverse} we conclude that $GC(G,S,\alpha)$ is a Cayley graph.

Suppose now that $G$ is not an elementary Abelian $2$-group and that the Sylow $2$-subgroup of $G$ is not cyclic. First consider the case when the Sylow $2$-subgroup of $G$ is elementary Abelian. Then, since $G$ is not elementary Abelian, the order of $G$ must be divisible by some odd number and hence $G=\ZZ_2\times \ZZ_2\times \ZZ_{2k+1}\times H$, where $k$ is a positive integer, and $H$ is an Abelian group.
Let $S_1=\{(1,0,0),(0,1,0),(1,1,1)\}$, $S_2=H\setminus\{1_H\}$, $S=S_1\times S_2$ and let $\iota_1$ be the restriction of $\iota$ to $\ZZ_2\times \ZZ_2\times \ZZ_{2k+1}$ and $\iota_2$ the restriction of $\iota$ to $H$. 
Then, by Lemma~\ref{lemma:directproduct}, $GC(G,S,\iota)\cong GC(\ZZ_2\times \ZZ_2\times \ZZ_{2k+1},S_1,\iota_1)\times GC(H,S_2,\iota_2)$. Since the direct product of two graphs is vertex-transitive if and only if both factors are vertex-transitive (see \cite[Theorem 8.19]{HIK}), and since by Example~\ref{ex:Z2 x Z2}, $GC(\ZZ_2\times \ZZ_2\times \ZZ_{2k+1},S_1,\iota_1)$ is not vertex-transitive, it follows that $GC(G,S,\iota)$ is not vertex-transitive, and consequently it is not a Cayley graph.

If the Sylow $2$-subgroup of $G$ is not elementary Abelian, then $G\cong \ZZ_{2^m}\times \ZZ_{2^n}\times H$, where $H$ is an Abelian group. Using Example~\ref{ex:2gr} and Lemma~\ref{lemma:directproduct} we construct a non-vertex-transitive generalized Cayley graph on $G$.
\end{proof}


\section{Generalized Cayley graphs of order $2p$}
\label{sec:2p}
\noindent

In this section we consider generalized Cayley graphs of order $2p$, where $p$ is a prime. Let $G$ be a group of order $2p$. Then $G\cong \ZZ_{2p}$ or $G\cong D_{2p}$. In the following lemma we consider the cyclic groups of order $2p$.
\begin{lemma}
\label{lem:cyc2p}
Every generalized Cayley graph on the group $\ZZ_{2p}$ is a Cayley graph.
\end{lemma}
\begin{proof}
If $G\cong \ZZ_{2p}$, then there are only two automorphisms of $G$ of order $2$, namely the identity mapping, or the inversion automorphism $\iota$.
For  $\alpha=1$,   $GC(G,S,\alpha)\cong \Cay(G,S)$ by definition. 
For $\alpha=\iota$, $GC(\ZZ_{2p},S,\alpha)$ is a Cayley graph on 
the group $D_{2p}\cong \ZZ_p\rtimes \ZZ_{2}$
by Theorem~\ref{thm:AbelInverse}.
\end{proof}

In the following lemma we consider the dihedral group of order $2p$, $D_{2p}=\langle \tau, \rho\ |\ \tau ^2=\rho ^p=1,\ \tau \rho \tau =\rho ^{-1}\rangle$.
\begin{lemma}
\label{lem:dih2p}
Every generalized Cayley graph on the dihedral group $D_{2p}$ is a Cayley graph.
\end{lemma}
\begin{proof}
The case when $p=2$ is trivial, therefore we will assume that $p>2$.
Let $\alpha$ be an automorphism of $D_{2p}$ of order $2$.
Since automorphisms preserve the order of the elements, it follows that $\alpha (\rho)=\rho^k$, where $(k,p)=1$, and $\alpha (\tau)=\tau \rho ^l$. The fact that $\alpha$ is an involution gives us the following restrictions on $k$ and $l$:
$$
\rho =\alpha (\alpha (\rho))=\rho ^{k^2}\Rightarrow k\equiv \pm 1\pmod{p}$$
$$\tau =\alpha (\alpha (\tau))=\alpha (\tau \rho ^l)=\alpha (\tau)\alpha(\rho ^l)=\tau \rho ^{l(k+1)}\Rightarrow l(k+1)\equiv 0 \pmod{p}.$$
Moreover, if $k=1$ then $l=0$ and $\alpha$ is the identity.
Therefore, we can assume that $\alpha(\rho)=\rho^{-1}$ and $\alpha(\tau)=\tau\rho^{l}$ for some $l\in \ZZ_{p}$.

Since $p$ is odd  there exists $k\in \ZZ_p$ such that $l=2k$.
Consider now a generalized Cayley graph $GC(D_{2p},S,\alpha)$.
Definition~\ref{def:gen_cayley}(ii) implies that
$$\alpha ((\rho ^t)^{-1})\rho ^t=\rho^{2t}\notin S,\ \forall\ t\in \mathbb{Z}_p,$$
and hence each element of $S$ is of the form $\tau\rho^i$, for some $i\in \ZZ_p$.
Let $S'=\{i\in \ZZ_p\mid \tau\rho^{i}\in S\}$.
Definition~\ref{def:gen_cayley}(iii) implies that $S'=l-S'=2k-S'$, and so $S'-k=-(S'-k)$. Let $S_1=\{\tau\rho^{s'-k}\mid s'\in S'\}$.
Define $\varphi :D_{2p}\rightarrow D_{2p}$ by
\begin{eqnarray*}
\varphi (\rho ^i)=\rho ^{-i-k},\\
\varphi (\tau \rho ^j)=\tau \rho ^j.
\end{eqnarray*}
We claim that $\varphi$ is isomorphism between $GC(D_{2p},S,\alpha)$ and $\Cay(D_{2p},S_1)$.
It is clear that $\varphi$ is a bijective mapping.
For an arbitrary edge $\{x,y\}$ of $GC(D_{2p},S,\alpha)$ we may,  
without loss of generality, assume that $x=\rho^a$ and $y=\tau \rho^{a+s'} $ for some $s'\in S'$.
We obtain that $\varphi(x)=\varphi(\rho^a)=\rho^{-a-k}$ and $\varphi(y)=\varphi(\tau \rho^{a+s'})=\tau \rho^{a+s'}=\varphi(x)\tau\rho^{s'-k}=\varphi(x)s_1$, where $s_1=\tau\rho^{s'-k}\in S_1$. Hence $\varphi$ preserves edges and we conclude that $\varphi$ is an isomorphism.
\end{proof}

Lemmas~\ref{lem:cyc2p} and \ref{lem:dih2p} imply the following theorem.
\begin{theorem}
Every generalized Cayley graph of order $2p$ is a Cayley graph.
\end{theorem}


\section{Unworthy generalized Cayley graphs}\label{sec:unworthy}
\noindent

Recall that a graph $X$ is said to be {\em unworthy} if there exist two vertices of $X$ with the same neighbourhood in $X$, and worthy otherwise. This section deals with the question which generalized Cayley graphs are unworthy.
In order to answer this question, let 
$$
K=\{\alpha(g)\mid gS=S\}=\{g\mid \alpha(g)S=S\}.
$$
where $G$ is a group, and $S\subseteq G$ and $\alpha\in \Aut(G)$ are such that they satisfy conditions in Definition~\ref{def:gen_cayley}.
Observe that $K$ is a subgroup of $G$.

\begin{proposition}\label{prop:unworthy}
The  vertices $x$ and $y$ in $X=GC(G,S,\alpha)$ have the same neighbours in $X$ if and only if $x^{-1}y\in K$.
\end{proposition}
\begin{proof}
Suppose first that $x^{-1}y\in K$, that is, $y=xk$ for some $k\in K$. Then the neighbourhood of $x$ is $\alpha(x)S$ and the neighbourhood of $y$ is $\alpha(y)S=\alpha(xk)S=\alpha(x)\alpha(k)S=\alpha(x)S$. Therefore $x$ and $y$ have the same neighbours. 

Conversely, if $x$ and $y$ have the same neighbours, then $\alpha(x)S=\alpha(y)S$, which implies $S=\alpha(x^{-1}y)S$. Therefore, $x^{-1}y\in K$, and the result follows.
\end{proof}
\begin{corollary}
The graph $GC(G,S,\alpha)$ is unworthy if and only if $K\neq \{1_G\}$.
\end{corollary}

The following proposition shows that an unworthy generalized Cayley graph can be decomposed into the lexicographic product of a worthy graph and an empty graph. (The {\it lexicographic product} of graphs $X$ and $Y$ is the graph $X[Y]$ with vertex set $V(X)\times V(Y)$, where two vertices $(x_1,y_1)$ and $(x_2,y_2)$ are adjacent if and only if either $\{x_1,x_2\}\in E(X)$ or $x_1 = x_2$ and $\{y_1,y_2\}\in E(Y)$.)
Let $X_K$ be the quotient graph of $X$ with respect to the partition $\{xK\mid x \in G\}$. 

\begin{proposition}\label{prop:lexigraphic}
Let $X=GC(G,S,\alpha)$ be unworthy. Then $X\cong X_K [\overline{K_n}]$, where $n=|K|$.
\end{proposition}
\begin{proof}
Suppose that the generalized Cayley graph $X=GC(G,S,\alpha)$ is unworthy. 
By Proposition~\ref{prop:unworthy}, two vertices of $X$ have the same neighbours if and only if they belong to the same left coset of $K$ in $G$.
Then $S$ is a union of several left cosets of $K$, and all the vertices in the same left coset of $K$ have the same neighbours. It is now easy to see that $X\cong X_K [\overline{K_n}]$.
\end{proof}
\begin{corollary}
If $G$ is an Abelian group and $S=G\setminus \omega_{\alpha}(G)$  then $GC(G,S,\alpha)\cong K_m [\overline{K_{n}}]$, where $n=\omega_{\alpha}(G)$, and $m=|G|/n$.
\end{corollary}
\begin{proof}
Since $G$ is Abelian, by Proposition~\ref{prop:claims}(ii), $\omega_{\alpha}(G)$ is a subgroup of $G$. 
We have $K=\{\alpha(g)\mid g\cdot (G\setminus \omega_{\alpha}(G))=G\setminus \omega_{\alpha}(G) \}=\{\alpha(g)\mid g\cdot \omega_{\alpha}(G)=\omega_{\alpha}(G) \}=\{\alpha(g)\mid g\in \omega_{\alpha}(G) \}$.
By Proposition~\ref{prop:claims}(iii), $\alpha(g)=g^{-1}$ for each $g\in \omega_{\alpha(G)}$, and therefore $K=\omega_{\alpha}(G)$.

We claim that $X_K\cong K_m$. The number of vertices in $X_K$ is equal to the index of $K$ in $G$, which is equal to $|G|/|K|=|G|/n=m$. Let $xK$ and $yK$ be two different vertices of $X_K$, that is $x^{-1}y\not\in K$. This implies that $\alpha(x)$ and $y$ are adjacent in $X$. Then, since $G$ is Abelian, $\alpha(x)= x\alpha(x)x^{-1}\in xK$. Therefore, vertices $xK$ and $yK$ are adjacent in $X_K$, and consequently $X_K\cong K_m$. The claim now follows by Proposition~\ref{prop:lexigraphic}.
\end{proof}

\section*{Acknowledgment}
The work of A. H. is supported in part by the Slovenian Research Agency (research program P1-0285 and research projects N1-0032, and N1-0038). The work of K. K. is supported in part by the Slovenian Research Agency (research program P1-0285 and research projects N1-0032, N1-0038, J1-6720, and J1-6743), 
in part by WoodWisdom-Net+, W$^3$B, and in part by NSFC project 11561021. 
The work of P. P. is supported in part by the Slovenian Research Agency (research program P1-0285 and Young Researcher Grant).

\end{document}